\documentclass{elsarticle}
 \usepackage{graphicx,amssymb,amstext,amsmath}
 \usepackage{bigstrut}
 \usepackage{enumerate}
 \usepackage{amsthm}
 \usepackage[english]{babel}
 \usepackage[all]{xy}
 \usepackage[vcentermath]{youngtab}
 \newtheorem{theorem}{Theorem}[section]
 \newtheorem{thm}[theorem]{Theorem}
 
 \newtheorem{lem}[theorem]{Lemma}
 \newtheorem{cor}[theorem]{Corollary}
 \newtheorem{prop}[theorem]{Proposition}
 \DeclareMathOperator\alt{Alt}
 \DeclareMathOperator\sym{Sym}
\DeclareMathOperator\PGL{PGL}
\DeclareMathOperator\hl{hl}

\begin{document}
\title{A new proof for the Erd\H{o}s-Ko-Rado Theorem for the alternating group}

 \author{Bahman Ahmadi}
 \ead{ahmadi2b@uregina.ca}

 \author{Karen Meagher \corref{cor1}\fnref{fn1}}
 \ead{karen.meagher@uregina.ca}
 \cortext[cor1]{Corresponding author}
 \fntext[fn1]{Research supported by NSERC.}

 \address{Department of Mathematics and Statistics,\\
 University of Regina, 3737 Wascana Parkway, S4S 0A4 Regina SK, Canada}

 \begin{abstract}
   A subset $S$ of the alternating group on $n$ points is {\it
     intersecting} if for any pair of permutations $\pi,\sigma$ in
   $S$, there is an element $i\in \{1,\dots,n\}$ such that
   $\pi(i)=\sigma(i)$. We prove that if $S$ is intersecting, then
   $|S|\leq \frac{(n-1)!}{2}$. Also, we prove that if $n \geq 5$, then
   the only sets $S$ that meet this bound are the cosets of the
   stabilizer of a point of $\{1,\dots,n\}$.
 \end{abstract}


\maketitle

\section{Introduction}
The famous Erd\H{o}s-Ko-Rado theorem~\cite{MR0140419} (abbreviated EKR
Theorem) gives bounds for the sizes of intersecting set systems and
characterizes the systems that achieve the bound. Many similar theorems
have been proved for other mathematical objects with a relevant
concept of ``intersection''.  See 
\cite{ MR2009400, MR1722210,MR0382015,MR2156694,MR657053} for versions
of this result for permutations, integer sequences, vector spaces, set
partitions and blocks in a $t$-design.

Let $G\leq \sym(n)$ be a permutation group with the natural action on
the set $\{1,\ldots, n\}$. Two permutations $\pi,\sigma\in G$ are said
to {\it intersect} if $\pi\sigma^{-1}$ has a fixed point. A subset
$S\subseteq G$ is, then, called {\it intersecting} if any pair of its
elements intersect. Clearly, the stabilizer of a point is an
intersecting set in $G$ (as is any coset of the stabilizer of a
point).  We say the group $G$ has the {\it EKR property}, if the size
of any intersecting subset of $G$ is bounded above by the size of the
largest point-stabilizer in $G$. Further, $G$ is said to have the {\it
  strict EKR property} if the only maximum intersecting subsets of $G$
are cosets of the stabilizer of a point. In \cite{MR2009400}, it was
proved that $\sym(n)$ has the strict EKR property. This result caught
the attention of several researchers, indeed, the result was proved
with vastly different methods in \cite{Meagher_Godsil,MR2061391} and
\cite{MR2419214}. Further, researchers have also worked on finding
other subgroups of $\sym(n)$ that have the strict EKR property. For
example in~\cite{KuW07} it is shown that $\alt(n)$ has the strict EKR
property, provided that $n\geq 5$.

In \cite{Meagher_Godsil}, the authors prove that the symmetric group
has the strict EKR property using an algebraic method which relies
strongly on the character theory of the symmetric group. This method
establishes an interesting connection between the algebraic properties
of the group and properties of a graph (based on the group) which can
be used to determine if a version of the EKR theorem holds. A more
significant characteristic of this approach is that it introduces a
standard way to determine if a permutation group has the strict EKR
property.  For instance, using this method, in \cite{MeagherS11} it is
proved that the projective general linear group $\PGL(2,q)$ acting on
the points of the projective line has the strict EKR property.  This
motivated the authors of the present paper to provide an alternative
prove of the strict EKR property for the alternating group using the
method of \cite{Meagher_Godsil}. In other words, we prove the
following.

\begin{thm}\label{main_Alt}
For $n\geq 5$, any intersecting subset of $\alt(n)$ has size at most
\[
\frac{(n-1)!}{2}.
\]
An intersecting subset of $\alt(n)$ achieves this bound if and only if
it is a coset of point-stabilizer.
\end{thm}

Throughout the paper we denote $G_n=\alt(n)$. The next section
provides a brief overview of the method used and explains how this
problem can be stated as a question about a graph. In
Section~\ref{rep_theory} a short background of the representation
theory of the symmetric group and the alternating group is
presented. Section~\ref{sec:twolayerhooks} proves some lower bounds on
the dimensions of some special representations of $G_n$ which will be
used in Section~\ref{standard_module}. In
Section~\ref{standard_module} we give more details about the standard
representation of the alternating group and its corresponding
module. In Section~\ref{main_proof} we give the proof of the main
theorem. We conclude with Section~\ref{conclusions} in which it is
proved that the strict EKR property for the $\sym(n)$ can be deduced
from the fact that $G_n$ has the strict EKR property.

\section{Overview of The Method}

This section is devoted to explaining the method for the proof of
Theorem~\ref{main_Alt}. To this goal, we start with recalling the
well-known {\it clique-coclique bound}; the version we use here was
originally proved by Delsarte \cite{MR0384310}. Assume $\mathcal{A}=\{
A_0, A_1,\ldots, A_d\}$ is an association scheme on $v$ vertices and
let $\{E_0, E_1,\ldots, E_d\}$ be the corresponding idempotents. (For a
detailed discussion about association schemes, the reader may refer to
\cite{MR2047311} or \cite{MR882540}.) For any subset $S$ of a set $X$,
we denote the characteristic vector of $S$ in $X$ by $v_S$.

\begin{thm}(Clique-Coclique Bound)\label{clique_coclique_baound} Let
  $X$ be the union of some of the graphs in an association scheme
  $\mathcal{A}$ on $v$ vertices. If $C$ is a clique and $S$ is an
  independent set in $X$, then
\[
|C||S|\leq v.
\]
If equality holds then
\[
v_C^T\,E_j\,v_C\,\,v_S^T\,E_j\,v_S = 0,\quad \text{for all} \quad j > 0.
\]
\end{thm}
We refer the reader to \cite{Meagher_Godsil} for a proof of
Theorem~\ref{clique_coclique_baound}. We will also make use of the
following straight-forward corollary of this result that was also
proved in \cite{Meagher_Godsil}.

\begin{cor}\label{clique_vs_coclique}
  Let $X$ be a union of graphs in an association scheme such
  that the clique-coclique bound holds with equality in $X$.  Assume
  that $C$ is a maximum clique and $S$ is a maximum independent set in
  $X$. Then, for $j > 0$, at most one of the vectors $E_jv_C$ and
  $E_jv_S$ is not zero.
\end{cor}

Let $G$ be a group and $D$ a subset of $G$, which does not include
the identity element of $G$ and is closed under inversion. The {\it
  Cayley graph of $G$ with respect to $D$} is defined to be the graph
$\Gamma(G;D)$ with vertex set $G$ in which two vertices $g,h$ are
adjacent if and only if $gh^{-1}\in D$. If $\mathcal{D}_G$ is
the set of all fixed-point-free elements of $G$, then the graph
$\Gamma(G,\mathcal{D}_G)$ is called the {\it derangement graph} of $G$
and is denoted by $\Gamma_G$. 

If we view $G$ as a permutation group, two permutations in $G$ are
intersecting if and only if their corresponding vertices are not
adjacent in $\Gamma_G$. Therefore, the problem of classifying the
maximum intersecting subsets of $G$ is equivalent to characterizing
the maximum independent sets of vertices in $\Gamma_G$. Since
$\mathcal{D}_G$ is a union of conjugacy classes of $G$ (namely the
derangement conjugacy classes), $\Gamma_G$ is a union of graphs in the
conjugacy class scheme of $G$. Thus the clique-coclique bound can be
applied to $\Gamma_G$.

Furthermore, the idempotents of conjugacy class scheme are $E_\chi$,
where $\chi$ runs through the set of all irreducible characters of $G$; the
entries of $E_\chi$ are given by
\begin{equation}\label{idempotent}
(E_\chi)_{\pi,\sigma}=\frac{\chi(1)}{|G|}\chi(\pi^{-1}\sigma).
\end{equation}
(see~\cite{MR546860} or \cite[Sections 2.2 and 2.7]{MR882540} for a
proof of this). The vector space generated by the columns of $E_\chi$
is called the {\it module} corresponding to $\chi$ or simply the {\it
  $\chi$-module} of $\Gamma_G$. For any character $\chi$ of $G$ and
any subset $X$ of $G$ define
\[
\chi(X)=\sum_{x\in X}\chi(x).
\]
Using Corollary~\ref{clique_vs_coclique} and Equation
(\ref{idempotent}) one observes the following.
\begin{cor}\label{at_most_one_non-zero}
  Assume the clique-coclique bound holds with equality for the graph
  $\Gamma_G$ and let $\chi$ be an irreducible character of $G$ that is
  not the trivial character. If there is a clique $C$ of maximum size
  in $\Gamma_G$ with $\chi(C)\neq 0$, then
\[
E_\chi\,v_S = 0
\]
for any maximum independent set $S$ of $\Gamma_G$.
\end{cor}
In other words, provided that the clique-coclique bound holds with
equality, for any module of $\Gamma_G$ (other than the trivial module)
the projection of at most one of the vectors $v_C$ and $v_S$ will be
non-zero, where $S$ is any maximum independent set and $C$ is any
maximum clique.

The outline for the proof of Theorem~\ref{main_Alt} is the
following:
\begin{enumerate}
\item Determine the irreducible characters of $G_n$ (Section~\ref{rep_theory}).
\item Show that clique-coclique bound holds with equality for
  $\Gamma_{G_n}$ (Section~\ref{standard_module}).

\item Find a maximum clique $C$ of $\Gamma_{G_n}$ such that $\chi(C)\neq
  0$, for any character $\chi$ which is not the standard character.
  (Section~\ref{standard_module}).

\item Show that the characteristic vector of any maximum independent set of
  $\Gamma_{G_n}$ lies in the direct sum of the trivial and the standard
  modules (Section~\ref{standard_module}).

\item Find a basis for the standard module---this basis is made up of
  characteristic vectors of cosets of point stabilizers
  (Section~\ref{standard_module}).

\item Complete the proof by showing that the only linear combination
  of the basis vectors that gives the characteristic vector for a
  maximum independent set, is the characteristic vector for the coset
  of the stabilizer of a point (Section~\ref{main_proof}).

\end{enumerate}

We conclude this section with the following lemma which will be needed
in Section~\ref{main_proof}. This proof was originally done by Mike
Newman (but has not been published elsewhere). The {\it eigenvalues of
a graph} are the eigenvalues of the adjacency matrix of the graph.

\begin{prop}\label{w}
  Let $X$ be a $k$-regular graph and let $\tau$ be the least
  eigenvalue of $X$. Assume that there is a collection $\mathcal{C}$ of
  cliques of $X$ of size $w$, such that every edge of $X$ is contained
  in a fixed number of elements of $\mathcal{C}$. Then
\[
\tau\geq -\frac{k}{w-1}.
\]
\end{prop}
\begin{proof}
  Assume that every edge of $X$ is contained exactly in $y$ cliques in
  $\mathcal{C}$. Then every vertex of $X$ is contained exactly in
  $\frac{k}{w-1}y$ cliques. Define a $01$-matrix $N$ as follows: the
  rows of $N$ are indexed by the vertices of $X$ and the columns are
  indexed by the members of $\mathcal{C}$; the entry $N_{(x,C)}$ is 1
  if and only if the vertex $x$ is in the clique $C$. We will,
  therefore, have
\[
NN^T=\frac{yk}{w-1}I+yA(X),
\]
where $I$ is the identity matrix and $A(X)$ is the adjacency matrix  of $X$. Thus
\[
\frac{k}{w-1}I+A(X)=\left(\frac{1}{\sqrt{y}}N\right)\left(\frac{1}{\sqrt{y}}N\right)^T,
\]
which implies that the matrix
\[
A(X)-\frac{-k}{w-1}I
\]
is positive semi-definite and, therefore, if $\tau$ is the least eigenvalue of $X$, then 
\[
\tau\geq \frac{-k}{w-1}.\qedhere
\]
\end{proof}

\section{Background on the representation theory of the symmetric
  group}\label{rep_theory}

In this section we provide a brief overview of the facts from the
representation theory of the symmetric group that we will make use of
in this paper. The reader may refer to any of the books
\cite{MR1153249,MR644144} and \cite{MR1824028} for more a detailed
discussion of this topic.

Let $G$ be a group and let $V$ be a representation of $G$ with the
character $\chi$. For any subgroup $H\leq G$, the {\it restriction} of
$V$ to $H$, denoted by $V\downarrow_H^G$, is the representation of $H$
with the character $\chi\downarrow_H^G$, where
\[
\chi\downarrow_H^G(h)=\chi(h),\quad h\in H.
\]

A weakly decreasing sequence $\lambda=[\lambda_1,\ldots, \lambda_k]$
of positive integers is called a {\it partition} of $n$ if $\sum
\lambda_i=n$. If $\lambda$ is a partition of $n$ then we write
$\lambda \vdash n$.  To any partition
$\lambda=[\lambda_1,\ldots,\lambda_k]$ of $n$, we associate a {\it
  Young diagram}, which is an array of $n$ boxes having $k$
left-justified rows with row $i$ containing $\lambda_i$ boxes, for
$1\leq i \leq k$. The {\it transpose} (or {\it conjugate}) of a
partition $\lambda\vdash n$, which is denoted by $\hat{\lambda}$, is
the partition corresponding to the Young diagram which is obtained
from that of $\lambda$ by interchanging rows and columns. A partition
is {\it symmetric} if it is equal to its own conjugate.
Figure~\ref{young_example} displays the Young diagram of the partition
$\lambda=[5,3,3,2,1,1]$ of $15$ and its transpose.

\begin{figure}[h!]
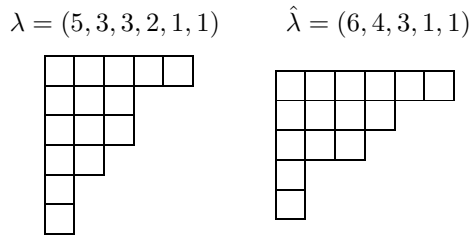

\begin{center}
\begin{tabular}{cc}
$\lambda=(5,3,3,2,1,1)$ \qquad & \quad $\hat{\lambda}=(6,4,3,1,1)$\\ [.2cm]
$\yng(5,3,3,2,1,1)$ & $\yng(6,4,3,1,1)$
\end{tabular}
\caption{An example of a Young diagram and its conjugate.}
\label{young_example}
\end{center}
\end{figure}

It is well-known that there is a one-to-one correspondence between the
partitions of $n$ and the irreducible representations of
$\sym(n)$. Indeed, any irreducible representation of $\sym(n)$ is of
the form $S^\lambda$ where $\lambda$ is a partition of $n$.  The
spaces $S^{\lambda}$, which are called the {\it Specht modules}, are
$\mathbb{C}$-algebras generated by $\lambda$-polytabloids (see
\cite{MR1824028} for more details). 

Our next result is the {\it hook-length formula}; this
result states a way to evaluate the dimension of the representation $S^\lambda$ of
$\sym(n)$, based on properties of the partition $\lambda$. Its
proof relies on the Frobenius formula (see \cite{MR1153249}).  For any
box of the Young diagram of $\lambda$, the corresponding {\it hook
  length} is one plus the number of boxes horizontally to the right
and vertically below the box.  Define $\hl(\lambda)$ to be the product
of all hook lengths of $\lambda$.

\begin{lem}\label{hl_formula}
  If $\lambda\vdash n$, then dimension of the irreducible
  representation of $\sym(n)$ corresponding to $\lambda$ is
  $n!/\hl(\lambda)$.
\end{lem}

 The following theorem establishes a connection
between the irreducible representations of $G_n$ and those of the
symmetric group. A general form of this theorem has been proved in
\cite[Section 5.1]{MR1153249}.  (Throughout this paper, we will use
the notation of Theorem~\ref{reps_of_alt}.)

\begin{thm}\label{reps_of_alt} Let $\lambda$ be a partition of $n$ and
  let $W$ and $\widehat{W}$ be the restrictions of $S^\lambda$ and
  $S^{\hat{\lambda}}$ to $G_n$, respectively. Then
\begin{enumerate}[(a)]
\item if $\lambda$ is not symmetric, then $W$ is an irreducible
  representation of $G_n$ and is isomorphic to $\widehat{W}$; and
\item if $\lambda$ is symmetric, then $W=W'\oplus W''$, where $W'$ and
  $W''$ are irreducible but not isomorphic representations of $G_n$.
\end{enumerate}
All the irreducible representations of $G_n$ arise uniquely in this way.
\end{thm}

For any conjugacy class $c$ of $G_n$, either $c$ is also a conjugacy
class in $\sym(n)$ or $c\cup c'$ is a conjugacy class in $\sym(n)$,
where $c'=tct^{-1}$, for some $t\notin G_n$.  The second type of
conjugacy classes are said to be {\it split}. A conjugacy class $c$ of
$G_n$ is split if and only if all the cycles in the cycle
decomposition of an element of $c$ have odd lengths and no two cycles
have the same length.

Suppose $c$ is a conjugacy class of $\sym(n)$ that is not a conjugacy
class in $G_n$. Assume that the decomposition of an element of $c$
contains cycles of odd lengths $q_1>q_2>\cdots>q_r$. Then we say $c$
{\it corresponds} to the symmetric partition
$\lambda=[\lambda_1,\lambda_2,\ldots]$ of $n$ if $q_1=2\lambda_1-1$,
$q_2=2\lambda_2-3$, $q_3=2\lambda_3-5,\ldots$. This is a
correspondence between a split conjugacy classes of $G_n$ and the
symmetric partitions  of $n$. See Figure~\ref{correspondence}.
\begin{figure}[h!]
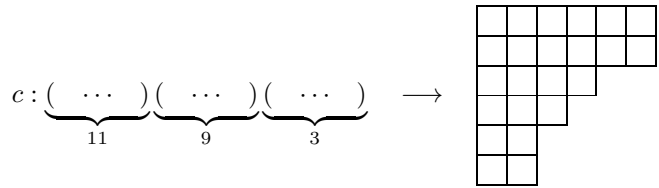

\centering
\[ c:\underbrace{(\quad\cdots\quad)}_{11} \underbrace{(\quad\cdots\quad)}_{9} 
  \underbrace{(\quad\cdots\quad)}_{3}\quad\longrightarrow \quad \yng(6,6,4,3,2,2)\]
\caption{\small The correspondence between split conjugacy classes and symmetric partitions}
\label{correspondence}
\end{figure}
Using this correspondence, we can give equations for the irreducible
characters of $G_n$ in terms of the characters of $\sym(n)$.  This
result is also proved in Section 5.1 of \cite{MR1153249}.

\begin{thm}\label{char_values_of_alt}
  Let $\lambda$ be a partition of $n$ and let $\chi^\lambda$ be the
  character of $S^\lambda$. Assume $c$ is a non-split conjugacy class
  of $G_n$ and $c'\cup c''$ is a pair of split conjugacy classes in
  $G_n$. Let $\sigma\in c$, $\sigma'\in c'$, $\sigma''\in c''$ and
  $\bar{\sigma}\in c'\cup c''$.
\begin{enumerate}[(a)]

\item If $\lambda$ is not symmetric, let $\chi_\lambda$ be the
  character of $W$, then
\[ 
\chi_\lambda(\sigma)=\chi^\lambda(\sigma)\quad\text{and}\quad 
 \chi_\lambda(\sigma')=\chi_\lambda(\sigma'')=\chi^\lambda(\bar{\sigma}).
\]

\item If $\lambda$ is symmetric, let $\chi_\lambda'$ and
  $\chi_\lambda''$ be the characters of $W'$ and $W''$, respectively, then
\[\chi_\lambda'(\sigma)=\chi_\lambda''(\sigma)=\frac{1}{2}\chi^\lambda(\sigma),\] and
\begin{enumerate}[(i)]
\item if $c'\cup c''$ does not correspond to $\lambda$ then \[
  \chi_\lambda'(\sigma')= \chi_\lambda'(\sigma'')=
  \chi_\lambda''(\sigma')=\chi_\lambda''(\sigma'')=\frac{1}{2}\chi^\lambda(\bar{\sigma}).\]
\item if $c'\cup c''$ corresponds to $\lambda$, then
\[ 
\chi_\lambda'(\sigma')= \chi_\lambda''(\sigma'')=x\quad\text{and}\quad \chi_\lambda'(\sigma'')=\chi_\lambda''(\sigma')=y.
\]
\end{enumerate}
The values of $x$ and $y$ are
\[
\frac{1}{2}\left[(-1)^m\pm\sqrt{(-1)^m q_1\dotsm q_r}\right],
\]
where $m=\frac{n-r}{2}$ and the cycle decomposition of an element of $c'\cup c''$ has cycles of odd lengths $q_1,\ldots,q_r$.
\end{enumerate}
\end{thm}

We will use the notation of Theorem~\ref{char_values_of_alt}
throughout this paper and hence want to emphasize that for
representations of $\sym(n)$ we use $\lambda$ as a superscript and for
representations of $\alt(n)$, the $\lambda$ is a subscript.

The next theorem is known as the {\it Murnaghan-Nakayama Rule}; it
gives a recursive way to determine the value of a character on a
conjugacy class. Before we can state this result, we need to define
some terms. The {\it $(i,j)$-block} in a Young diagram is the
block in the $i$-th row (from the top) and the $j$-th column (from the
left). If a Young diagram contains an $(i,j)$-block but not a
$(i+1,j+1)$-block then the $(i,j)$-block is part of what is called the
{\it boundary} of the Young diagram. A {\it skew hook} of $\lambda$ is
an edge-wise connected (meaning that all blocks are either side
by side or one below the other) part of the boundary blocks with the
property that removing them leaves a smaller proper Young diagram.
The length of a skew hook is the number of blocks it contains.

\begin{thm}\label{nakayama}
  If $\lambda \vdash n$ and $\sigma \in \sym(n)$ can be written as a
  product of an $m$-cycle and a disjoint permutation $h\in \sym(n-m)$,
  then
\[
\chi^{\lambda}(\sigma)=\sum_{\mu} (-1)^{r(\mu)}\chi^{\mu}(h),
\]
where the sum is over all partitions $\mu$ of $n-m$ that are obtained
from $\lambda$ by removing a skew hook of length $m$, and $r(\mu)$ is
one less than the number of rows of the removed skew hook.
\end{thm}
For a proof of this theorem, the reader may refer to \cite{MR1824028}.
We will state, without proof, two simple applications of this rule.

\begin{cor}\label{cor:appofMN}
Let $\sigma$ is an $n$-cycle in $\sym(n)$, then 
\[
\chi^\lambda(\sigma) =
 \begin{cases}
 (-1)^{n-\lambda_1},   & \textrm{if $\lambda=[\lambda_1,1^{n-\lambda_1}]$};\\
  0,                 &\textrm{otherwise.} 
 \end{cases}
\]
\end{cor}

Partitions $\lambda\vdash n$ of the form
$\lambda=[\lambda_1,1^{n-\lambda_1}]$ are called {\it hooks},
similarly partitions of the form $[\lambda_1,2,1^{n-\lambda_1-2}]$,
for $\lambda_1>1$ are called {\it near hooks}.

\begin{cor}\label{cor:appofMNtwo}
Let $\sigma$ be the product of two disjoint $n/2$-cycles in $\sym(n)$, then 
\[
\chi^\lambda(\sigma) \in \{0,\pm 1,\pm 2\}.
\]
\end{cor}

In the following section we will define a new type of partition and
show that these new partitions, along with near-hooks, are the only
partitions for which $\chi^\lambda(\sigma)$ in the above corollary
could be equal to $-2$.

\section{Two-Layer Hooks}
\label{sec:twolayerhooks}

Assume $\lambda=[\lambda_1,\ldots,\lambda_k]$ is a partition of $n$
such that $k\geq 3$, $\lambda_2+\hat{\lambda}_2\geq 5$, $\lambda_3\leq
2$ and $\lambda_1-\lambda_2=\hat{\lambda}_1-\hat{\lambda}_2>0$. Then
we say $\lambda$ is a {\it two-layer hook}.  In fact, a two-layer hook
is a partition whose Young diagram is obtained by ``appropriately
gluing'' two hooks of lengths greater than $1$.  See
Figure~\ref{two_layer_hook} for some examples of two-layer hooks.
\begin{figure}[h!]
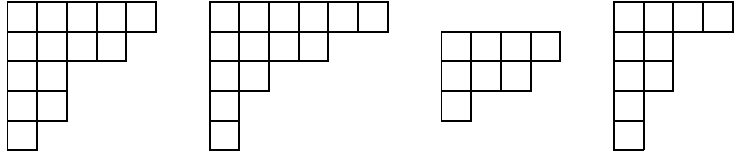

\[\yng(5,4,2,2,1)\quad\quad \yng(6,4,2,1,1)\quad\quad \yng(4,3,1)\quad\quad \yng(4,2,2,1,1) \]
\caption{\small Two-layer hooks}
\label{two_layer_hook}
\end{figure}
Note that if $\lambda \vdash n$ is a two-layer hook, then
$\hat{\lambda}$ is also a two-layer hook and $n$ must be at least
$8$. Note also that a near hook is not a two-layer hook.

\begin{lem}\label{char_of_two_layer_hook}
  Let $\lambda$ be a partition of $n$ and let $\sigma$ be a
  permutation in $\sym(n)$ that is the product of two disjoint
  $n/2$-cycles. If $\chi^\lambda(\sigma)=-2$, then $\lambda$ is either
  a two-layer hook or a symmetric near hook.
\end{lem}
\begin{proof}
  According to the Murnaghan-Nakayama Rule and
  Corollary~\ref{cor:appofMN}, $\lambda$ should have two
  skew-hooks of length $n/2$ and deleting each of them should leave a
  hook of length $n/2$. If we denote $\lambda =
  [\lambda_1,\ldots,\lambda_k]$, then this obviously implies that
  $k>1$.

  If $k=2$, then $\lambda$ must be the
  partition $[\frac{n}{2},\frac{n}{2}]$ (since if $\lambda=[\lambda_1,\lambda_2]$, where
  $\lambda_1>\lambda_2$, then $\lambda$ will not have two skew-hooks of
  length $n/2$); in this case we can calculate the character value at
  $\sigma$ to be $2$.  Thus $k\geq 3$.  

  If $\lambda_3>2$, then the partition $\lambda'$ obtained from
  $\lambda$ by deleting any skew-hook will have $\lambda'_2\geq 2$
  which implies that $\lambda'$ is not a hook. Thus $\lambda_3\leq 2$.

  Let $\lambda_1-\lambda_2=s$ and
  $\hat{\lambda}_1-\hat{\lambda}_2=t$. Assume $\mu$ and $\nu$ are the
  two skew hooks of $\lambda$ of length $n/2$. Since they have length
  $n/2$, we may assume that $\mu$ contains the last box of the first
  row and $\nu$ contains the last box of the first column. The lengths
  of $\mu$ and $\nu$ being both equal to $n/2$ implies that
  \[
  (s+1) +(\lambda_2-1) +(\hat{\lambda}_2-1)-1=
  (t+1)+(\hat{\lambda}_2-1)+(\lambda_2-1)-1,
  \]
  which yields $s=t$.

  If $s=t=0$ then $\lambda=[\lambda_1,\lambda_1, 2, \dots 2]$. If we
  denote the number of rows in $\lambda$ by $k$, then according to the
  Murnaghan-Nakayama Rule
\begin{align*}
\chi^\lambda(\sigma) &= (-1)^{r(\mu)} (-1)^{r(\lambda\setminus \mu)}\,+\, (-1)^{r(\nu)} (-1)^{r(\lambda\setminus \nu)}\\
&= (-1)^{k} (-1)^{k}\,+\, (-1)^{k-1} (-1)^{k-1}\\
&=2. 
\end{align*}

Finally, note that if $\lambda_2+\hat{\lambda}_2<5$, then either
$\lambda_2+\hat{\lambda}_2=2$ or $4$. In the former case, $\lambda$ is
a hook and obviously it cannot have two skew-hooks of length $n/2$. In
the latter case, $\lambda$ must be a near hook. If it is not symmetric
then it cannot have two skew-hooks.

These imply that if $\lambda$ is neither symmetric near hook nor a two layer
hook, then $\chi^\lambda(\sigma)\neq -2$; this completes the proof.
\end{proof}

The following lemma provides a lower bound on the dimension of a symmetric near hook.
\begin{lem}\label{dimensions_of_near_hooks}
If a symmetric partition $\lambda$ of $n\geq 8$ is a near hook, then  
 $\chi^\lambda(1)> 2n-2$.
\end{lem}
\begin{proof}
  Since $\lambda$ is a symmetric near hook
  we know that $\lambda=[n/2,2,1^{\frac{n}{2}-2}]$ and we can calculate the hook
  lengths directly
\begin{align*}
\hl(\lambda)&=(n-1)\left(\frac{n}{2}\,\right)^2\left[\left(\frac{n}{2}-2\right)!\right]^2\\[.2cm]
&\leq (n-1)\,\frac{n^2}{4}\,(n-4)!=\frac{n(n-1)}{2(n-2)(n-3)}\,\,\frac{n(n-2)!}{2}\\[.2cm]
&<\frac{n(n-2)!}{2}
\end{align*}
since $n\geq 8$. Putting this bound into the hook-length formula
(Lemma~\ref{hl_formula}) gives the lemma.
\end{proof}

Next we prove that the same lower bound holds for the dimension of a two-layer hook.
\begin{lem}\label{dimensions_of_2_layer_hooks}
If a partition $\lambda$ is a two-layer hook, then $\chi^\lambda(1)> 2n-2$.
\end{lem}
\begin{proof}
  Let $\lambda=[\lambda_1,\lambda_2,\ldots,\lambda_k]$. According to
  the hook-length formula, it suffices to show that
  $\hl(\lambda)<n(n-2)!/2$.  We proceed by induction on $n$. It is easy
  to see the lemma is true for $n=8$. Let $n\geq 10$ and without loss
  of generality assume that $\lambda_1>\hat{\lambda}_1$. This implies
  $\lambda_2\geq 3$. We compute
\begin{align*}
\hl(\lambda)=&(\lambda_1+\hat{\lambda}_1-1)(\lambda_1+\hat{\lambda}_2-2) (\lambda_2+\hat{\lambda}_1-2) (\lambda_2+\hat{\lambda}_2-3)\\ &\cdot\frac{(\lambda_1-1)!}{s+1}\frac{(\hat{\lambda}_1-1)!}{s+1}(\lambda_2-2)! (\hat{\lambda}_2-2)!
\end{align*}
where $s$ is as in Lemma~\ref{char_of_two_layer_hook}. One can re-write this as
\begin{align}
\hl(\lambda)
=& \frac{\lambda_1+\hat{\lambda}_1\!-\!1}{\lambda_1+\hat{\lambda}_1\!-\!2}\,\cdot\,\frac{\lambda_1+\hat{\lambda}_2\!-\!2}{\lambda_1+\hat{\lambda}_2\!-\!3} \,\cdot\,\frac{\lambda_2+\hat{\lambda}_1\!-\!2}{\lambda_2+\hat{\lambda}_1\!-\!3}\,\cdot\, \frac{\lambda_2+\hat{\lambda}_2\!-\!3}{\lambda_2+\hat{\lambda}_2\!-\!4} 
 \cdot(\lambda_1\!-\!1)(\lambda_2\!-\!2) \nonumber \\
& \left[\vphantom{\frac12}(\lambda_1+\hat{\lambda}_1-2)(\lambda_1+\hat{\lambda}_2-3)(\lambda_2+\hat{\lambda}_1-3)(\lambda_2+\hat{\lambda}_2-4)\right.\nonumber \\
& \quad \left.\cdot   \frac{(\lambda_1-2)!}{s+1}\frac{(\hat{\lambda}_1-1)!}{s+1}(\lambda_2-3)! (\hat{\lambda}_2-2)!\right]\nonumber \\
=& \frac{\lambda_1+\hat{\lambda}_1-1}{\lambda_1+\hat{\lambda}_1-2}\,\cdot\,\frac{\lambda_1+\hat{\lambda}_2-2}{\lambda_1+\hat{\lambda}_2-3} \,\cdot\,\frac{\lambda_2+\hat{\lambda}_1-2}{\lambda_2+\hat{\lambda}_1-3}\,\cdot\, \frac{\lambda_2+\hat{\lambda}_2-3}{\lambda_2+\hat{\lambda}_2-4} \nonumber \\ 
& \cdot(\lambda_1-1)(\lambda_2-2)  \hl(\tilde{\lambda}),\label{tilde_lambda}
\end{align}
where
$\tilde{\lambda}=[\lambda_1-1,\lambda_2-1,\lambda_3,\ldots,\lambda_k]$
is the partition whose Young diagram is obtained from that of
$\lambda$ by removing the last boxes of the first and the second rows.
We can simplify (\ref{tilde_lambda}) as
\begin{align*}
\hl(\lambda)=& \left(1+\frac{1}{\lambda_1+\hat{\lambda}_1-2}\right)\,\left(1+\frac{1}{\lambda_1+\hat{\lambda}_2-3} \right) \,\left(1+\frac{1}{\lambda_2+\hat{\lambda}_1-3}\right)\\
 &  \left(1+\frac{1}{\lambda_2+\hat{\lambda}_2-4}\right)(\lambda_1-1)(\lambda_2-2)\cdot \hl(\tilde{\lambda}).
\end{align*}
We now observe the following facts:
\begin{enumerate}
\item $\lambda_1+\hat{\lambda}_1-2> n/2$; thus
\[
1+\frac{1}{\lambda_1+\hat{\lambda}_1-2}\,<\,\frac{n+2}{n}.
\]
\item By Lemma~\ref{char_of_two_layer_hook} and the definition of a two-layer hook, $\lambda_1-\lambda_2=\hat{\lambda}_1-\hat{\lambda}_2$; hence $\lambda_1+\hat{\lambda}_2=\lambda_2+\hat{\lambda}_1$. On the other hand
\[
\lambda_1+\hat{\lambda}_2+\lambda_2+\hat{\lambda}_1-5=n-1,
\]
hence
\[
1+\frac{1}{\lambda_1+\hat{\lambda}_2-3}=\frac{n}{n-2},
\quad
1+\frac{1}{\lambda_2+\hat{\lambda}_1-3}=\frac{n}{n-2}.
\]
\item Since $\lambda_2+\hat{\lambda}_2\geq 5$, we have
\[
1+\frac{1}{\lambda_2+\hat{\lambda}_2-4}\,<\,2.
\]
\item Since $\lambda_1+\lambda_2\leq n-1$, we have
\[
(\lambda_1-1)(\lambda_2-2)\leq \frac{(n-4)^2}{4}.
\]
\end{enumerate}

The partition $\tilde{\lambda}$ is either a near hook or a two-layer
hook. In the first case, because $\lambda$ is a two-layer hook, we
have
\[
\tilde{\lambda}_1-2=\tilde{\lambda}_1-\tilde{\lambda}_2=\lambda_1-\lambda_2=\hat{\lambda}_1-\hat{\lambda}_2=\hat{\lambda}_1-2=\widehat{\tilde{\lambda}}_1-2,
\]
that is, the sizes of the first row and the first column of
$\tilde{\lambda}$ are equal which implies that $\tilde{\lambda}$ is
symmetric and, thus, according to
Lemma~\ref{dimensions_of_near_hooks},
\[
\hl(\tilde{\lambda})<\frac{(n-2)(n-4)!}{2}.
\]
If $\tilde{\lambda}$ is a two layer hook, then the same bound holds by
the induction hypothesis. Therefore
\begin{align*}
\hl(\lambda)\,&<\, 2\,\,\frac{n+2}{n}\, \frac{n}{n-2}\, \frac{n}{n-2}\, \frac{(n-4)^2}{4}\, \frac{(n-2)(n-4)!}{2}\\[.3cm]
&=\frac{n(n+2)(n-4)^2(n-4)!}{4(n-2)}\\[.3cm]
&=\frac{(n+2)(n-4)^2}{2(n-2)^2(n-3)}\,\,\frac{n(n-2)! }{2}\\[.3cm]
& <\,\frac{n(n-2)! }{2}. \qedhere
\end{align*}
\end{proof}

\section{The Standard Module}\label{standard_module}

The representation of $\sym(n)$ corresponding to $[n]$ is called the {\it trivial}
representation, the character for this representation is equal to $1$
for every permutation. If $\lambda= [n-1,1]$, then the irreducible
representation $S^\lambda$ of the $\sym(n)$ is called the {\it
  standard representation}. For $n\geq 5$, $\lambda=[n-1,1]$ is not
symmetric; hence the restriction $V$ of $S^\lambda$ to $G_n$ is also
irreducible. We also call this representation the {\it standard
  representation} of $G_n$ and $V$ is the standard module of
$G_n$. The value of the character of the standard representation on a
permutation $\sigma$ is the number of elements of $\{1,\dots,n\}$ fixed by
$\sigma$ minus $1$.

In this section we prove that the characteristic vector of any maximum
intersecting subset of $G_n$ is in the direct sum of the trivial
module and the standard module of $G$. To do this, we will show that
the clique-coclique bound holds with equality and, moreover, for each
$\lambda$ which is neither $[n]$ nor $[n-1,1]$, there is a clique $C$
of maximum size with $E_{\lambda}v_C \neq 0$. From this, using
Corollary~\ref{clique_vs_coclique}, we conclude that $E_{\lambda}v_S =
0$ for any maximum independent set, unless $E_{\lambda}$ is the
projection to either the trivial module or the standard module.
According to Corollary~\ref{at_most_one_non-zero}, it is sufficient to
show that for each irreducible representation $\chi$ of $G_n$ there is
a maximum clique $C$ in $\Gamma_{G_n}$ with $\chi(C) \neq 0$.  This is
very similar to what was done in \cite[Section 5]{Meagher_Godsil}.  To
do this, we will consider two cases, first when $n$ is odd and
second when it is even.

\subsection{Case 1: $n$ is odd:} 
\label{subsec:odd}

Assume $n\geq 5$ to be odd. Theorem 1.1 of \cite{AlspachGSV03} proves
that there is a decomposition of the arcs of the complete digraph
$K_n^\ast$ on $n$ vertices to $n-1$ directed cycles of length
$n$. Each of these cycles corresponds to an $n$-cycle in $G_n$. Since
no two such decompositions share an arc in $K_n^\ast$, no two of the
corresponding permutations intersect.  Let $C$ be the set of these
permutations together with the identity element of $G_n$. Then $C$ is
a clique in $\Gamma_{G_n}$ of size $n$.  The set of all $n$-cycles
from $\sym(n)$ form a pair of split conjugacy classes $c'_0$ and
$c''_0$ in $G_n$. Thus all the non-identity elements of $C$ lie in
$c'_0\cup c_0''$.

\begin{lem}\label{n_odd} Let $n\geq 5$ be odd. Then for any
  irreducible character $\chi$ of $G_n$, other than the standard
  character, we have $\chi(C)\neq 0$, where $C$ is the set defined above.
\end{lem}
\begin{proof}
  Theorem~\ref{char_values_of_alt} gives all the irreducible
  representations of $G_n$. First consider the case where $\chi$ is
  the character of the restriction of the representation
  $S^{\lambda}$, where $\lambda$ is not symmetric, to $G_n$. Let
  $\chi^\lambda$ be the character of $S^{\lambda}$ then $\chi =
  \chi_\lambda$. According to Theorem~\ref{char_values_of_alt},
  $\chi_\lambda$ has the same values on $c'_0$ and $c_0''$ and this
  value is equal to the value of $\chi^\lambda$ on $c'_0 \cup
  c''_0$. We compute
\[
\chi_\lambda(C)=\sum_{x\in C}\chi_\lambda(x)=\chi^\lambda(1)+(n-1)\chi^\lambda(\sigma),
\]
where $\sigma$ is a cyclic permutation of length $n$. Using the
corollary of the Murnaghan-Nakayama Rule
(Corollary~\ref{cor:appofMN}), we have $\chi^\lambda(\sigma)\in
\{0,\pm 1\}$. Therefore, if $\chi_\lambda(C)=0$, since
$\chi_\lambda(1) >0$, it must be that $\chi^\lambda(\sigma)= -1$ and
then $\chi^\lambda(1)=n-1$. The representations corresponding to the
partition $[n-1,1]$ and its transpose, $[2,1,\ldots,1]$, are the only
representations of $\sym(n)$ of dimension $n-1$ and according to
Theorem~\ref{reps_of_alt}, their restrictions to $G_n$ are both
isomorphic to the standard representation of $G_n$.

Next assume that $\chi$ is the character of one of the two irreducible
representations $W'$ or $W''$, where $W = W'\oplus W''$ is the
restriction of $S^\lambda$ to $G_n$; in this case $\lambda$ must be
symmetric.  Thus, using the notation of
Theorem~\ref{char_values_of_alt}, $\chi = \chi_\lambda'$ (the case
when $\chi = \chi_\lambda''$ is identical, so we omit it).  If $\lambda$
is not the hook $[(n+1)/2,1,\ldots,1]$, then according to
Theorem~\ref{char_values_of_alt}, we have
\[
\chi_\lambda'(C)=\sum_{x\in C}\chi_\lambda'(x)=\frac{1}{2}\chi^\lambda(1)+(n-1)\frac{1}{2}\chi^\lambda(\sigma).
\]
Thus, as in the previous case, if $\chi_\lambda'(C)=0$, then we must have $\chi^\lambda(1)=n-1$ which is a contradiction.

The final case that we need to consider is when $\chi$ is the
character of one of the two irreducible representations whose sum is
the representation formed by restricting $S^\lambda$ to $G_n$ where
$\lambda=[(n+1)/2,1,\ldots,1]$. Again we assume that $\chi =
\chi_\lambda'$ (since the case for $\chi = \chi_\lambda''$ is
identical) and using Theorem~\ref{char_values_of_alt}, we have
\begin{align*}
\chi_\lambda'(C)&=\sum_{x\in C}\chi_\lambda'(x)\\
&=\chi'_\lambda(1)+\sum_{x\in C\cap c'_0}\chi_\lambda'(x)+ \sum_{x \in C\cap c''_0}\chi_\lambda'(x)\\
&=\frac{1}{2}\chi^\lambda(1)+r'\, \frac{1}{2}\left[(-1)^\frac{n-1}{2}+\sqrt{(-1)^\frac{n-1}{2} \,n}\right] +r'' \,\frac{1}{2}\left[(-1)^\frac{n-1}{2}-\sqrt{(-1)^\frac{n-1}{2}\,n} \right],
\end{align*}
where $r'=|C\cap c'_0|$ and $r''=|C\cap c''_0|$. Note that $r'+r'' =n-1$.  Hence, if $\chi_\lambda'(C)=0$, then we must have
\begin{equation}\label{eq_on_dim_lambda}
-\chi^\lambda(1)=r'\left[(-1)^\frac{n-1}{2}+\sqrt{(-1)^\frac{n-1}{2} \,n}\right]\,+ \,r''\left[(-1)^\frac{n-1}{2}-\sqrt{(-1)^\frac{n-1}{2}\,n} \right].
\end{equation}
Note that
\[
\chi^\lambda(1)=\frac{2^{n-1}(n-2)!!}{(n-1)!!},
\]
where  $a!!=a(a-2)(a-4)\cdots 2$ if $a$ is even positive integer and $a!!=a(a-2)(a-4)\cdots 1$, if $a$ is odd.
Consider the following two cases. If $4\nmid n-1$, then (\ref{eq_on_dim_lambda}) implies that
\begin{equation}\label{not_divisible_by_4}
-\frac{2^{n-2}(n-2)!!}{(n-1)!!}=-(n-1)+\sqrt{-n}(r'-r'').
\end{equation}
It follows, then, that  $r'=r''$ and so
\[\frac{2^{n-1}(n-2)!!}{(n-1)!!}=n-1,\]
since this only holds for $n=3$, this is a contradiction.

On the other hand, if $4\mid n-1$, then (\ref{eq_on_dim_lambda}) implies that
\begin{equation}\label{divisible_by_4}
-\frac{2^{n-1}(n-2)!!}{(n-1)!!}=r'\, (1+\sqrt{n})\,\,+ \,\, r''\, (1-\sqrt{n});
\end{equation}
that  is,
\begin{align}
\frac{2^{n-1}(n-2)!!}{(n-1)!!} &=-(n-r''-1)\, (\sqrt{n}+1)\,\,+ \,\, r''\, (\sqrt{n}-1)  \nonumber \\
             &\leq (n-1)\, (\sqrt{n}-1)\leq n^{\frac{3}{2}}. \label{asymptotic_inequality}
\end{align}
Note that
\[
\frac{2^{n-1}(n-2)!!}{(n-1)!!}=\frac{2^{n-1}}{n}\frac{n!!}{(n-1)!!}>\frac{2^{n-1}}{n};
\]
thus (\ref{asymptotic_inequality}) yields
\[
2^{n-1}< n^\frac{5}{2}.
\]
It is easily seen that this inequality fails for all $n\geq
9$. Finally, note that (\ref{divisible_by_4}) and
(\ref{not_divisible_by_4}) lead us to contradictions if $n=5$ and if
$n=7$, respectively. This completes the proof of the lemma.
\end{proof}

Next we consider when $n$ is even

\subsection{Case 2: $n$ is even}
\label{subsec:even}

In this part, we assume $n\geq 6$ to be even. According to Theorem 1.1
in \cite{AlspachGSV03}, the arcs of the complete digraph $K_n^\ast$
can be decomposed to $n-1$ pairs of vertex-disjoint directed cycles of
length $n/2$. Each of these pairs corresponds to a permutation in
$G_n$ which is a product of two cyclic permutations of length $n/2$.
Let $C$ be the set of these permutations together with the identity
element of $G_n$. Then, similar to the previous part, $C$ is a clique
in $\Gamma_{G_n}$. Note that the non-identity elements of $C$ lie in a
non-split conjugacy class $c$ of $G_n$. Now we prove the equivalent of
Lemma~\ref{n_odd} for even $n$, using this set $C$.

\begin{lem}\label{n_even} Let $n\geq 6$ be even. Then for any
  irreducible character $\chi$ of $G_n$, which is not the standard
  character, we have $\chi(C)\neq 0$, where $C$ is as defined above.
\end{lem}
\begin{proof}

  First consider the case $\chi = \chi_\lambda$ where $\lambda$
  is not symmetric. Using the notation of
  Theorem~\ref{char_values_of_alt}, we have
\[
\chi_\lambda(C)=\sum_{x\in C}\chi_\lambda(x)=\chi^\lambda(1)+(n-1)\chi^\lambda(\sigma),
\]
where $\sigma$ is a product of two disjoint cyclic permutations of
length $n/2$. Now, suppose $\chi_\lambda(C)=0$. Then
\begin{equation}\label{even_nonsymmetric}
-\chi^\lambda(1)=(n-1)\chi^\lambda(\sigma).
\end{equation}

Using the Murnaghan-Nakayama Rule, we have
$\chi^\lambda(\sigma)\in\{0,\pm 1,\pm2\}$ (see
Corollary~\ref{cor:appofMNtwo}). If $\chi^\lambda(\sigma)=0, 1$ or $2$,
then Equation (\ref{even_nonsymmetric}) yields a contradiction with
the fact that $\chi_\lambda(1)$ is strictly positive. Also if
$\chi^\lambda(\sigma)=-1$, then we must have $\chi^\lambda(1)=n-1$
which contradicts the fact that the standard representation and its
conjugate are the only irreducible representations of $\sym(n)$ of
dimension $n-1$. Hence, suppose $\chi^\lambda(\sigma)=-2$.  Then
$\chi^\lambda(1)=2n-2$.  But according to
Lemma~\ref{char_of_two_layer_hook}, $\lambda$ must be a two-layer hook
or a symmetric near hook. Then by Lemma~\ref{dimensions_of_near_hooks}
and Lemma~\ref{dimensions_of_2_layer_hooks}, the dimension of $\chi$
is strictly greater than $2n-2$. 

Next consider the case where $\chi$ is the character of one of the two
irreducible representations in the restriction of the representation
$S^{\lambda}$ to $G_n$, where $\lambda$ is not symmetric; so $\chi =
\chi_\lambda'$ or $\chi_\lambda''$. We will show that
$\chi_\lambda'(C)\neq 0$; the proof that $\chi_\lambda''(C)\neq 0$ is
similar. We have
\[
\chi_\lambda'(C)=\sum_{x\in C}\chi_\lambda'(x)=\frac{1}{2}\chi^\lambda(1)+(n-1)\frac{1}{2}\chi^\lambda(\sigma),
\]
where $\sigma$ is a product of two disjoint $n/2$-cycles. If
$\chi_\lambda(C)=0$, then with the same argument as above, we get a
contradiction.
\end{proof}

We now prove the main theorem of this section.

\begin{prop}\label{max_indy_is_in_standard_module}
  Let $S$ be an intersecting subset of $G_n$ of size $(n-1)!/2$ and
  let $v_S$ be the characteristic vector of $S$. Then the vector
  $v_S-\frac{1}{n}\mathbf{1}$ is in the standard module of $G_n$.
\end{prop}
\begin{proof} Let $S_{1,1}$ be the point stabilizer for $1$ in $G_n$;
  so $S_{1,1}$ is an independent set of size $\frac{(n-1)!}{2}$ in $\Gamma_{G_n}$. Then the
  cliques defined in Lemma~\ref{n_odd} and Lemma~\ref{n_even},
  together with $S_{1,1}$ prove that the clique-coclique bound holds with
  equality for $\Gamma_{G_n}$.  Given any irreducible character $\chi$
  of $G_n$, except the standard character and the trivial character, according to
  Lemma~\ref{n_odd} and Lemma~\ref{n_even}, there is a maximum clique
  $C$, such that $\chi(C)\neq 0$. Hence, according to
  Corollary~\ref{at_most_one_non-zero}, we have $E_\chi v_S=0$, for
  any maximum independent set $S$. This implies that if $\chi$ is neither
  the trivial nor the standard character, then $E_\chi
  (v_S-\frac{\mathbf{1}}{n})=0$. It is not hard to see that the vector
  $v_S-\frac{\mathbf{1}}{n}$ is orthogonal to $\mathbf{1}$; hence it
  cannot lie in the trivial module of $G_n$. Therefore, for any
  maximum independent set $S$, the vector $v_S-\frac{1}{n}\mathbf{1}$
  belongs to the standard module of $G_n$.
\end{proof}

In the remainder of this section, we will provide a basis for the standard
module of $G_n$. For any pair $i,j\in \{1,\dots,n\}$, define $S_{i,j}$ to be the
set of all permutations $\pi\in G_n$ such that $\pi(i)=j$. Note that
$S_{i,j}$ are cosets of point stabilizers in $G_n$ under the natural
action of $G_n$ on $\{1,\dots,n\}$ and that they are maximum intersecting sets
in $G_n$. Define $v_{i,j}$ to be the characteristic vector of
$S_{i,j}$, for all $i,j\in \{1,\dots,n\}$.

\begin{lem}\label{basis}
The set
\[
B:=\{v_{i,j}-\frac{1}{n}\mathbf{1}\,|\, i,j\in[n-1]\} 
\]
is a basis for the standard module $V$ of $G_n$.
\end{lem}
\begin{proof}
  According to Proposition~\ref{max_indy_is_in_standard_module}, we
  have $B\subset V$ and since the dimension of $V$ is equal to
  $|B|=(n-1)^2$, it suffices to show that $B$ is linearly
  independent. Note, also, that since $\mathbf{1}$ is not in the span
  of $v_{i,j}$ for $i,j\in[n-1]$, it is enough to prove that the set
  $\{v_{i,j}\,|\, i,j\in[n-1]\}$ is independent.  

  Define a matrix $H$ to have the vectors $v_{i,j}$, with $i,j\in
  [n-1]$, as its columns.  Then the rows of $H$ are indexed by the
  elements of $G_n$ and the columns are indexed by the ordered pairs
  $(i,j)$, where $i,j\in [n-1]$; we will also assume that the ordered
  pairs are listed in lexicographic order.  It is easy to see that
\[
H^TH=\frac{(n-1)!}{2}\,I_{(n-1)^2}\,+\, \frac{(n-2)!}{2}\left( A(K_{n-1})\otimes A(K_{n-1})\right),
\]
where $I_{(n-1)^2}$ is the identity matrix of size $(n-1)^2$ and
$A(K_{n-1})$ is the adjacency matrix of the complete graph
$K_{n-1}$. The distinct eigenvalues of $A(K_{n-1})$ are $-1$ and
$n-2$; thus the eigenvalues of $A(K_{n-1})\otimes A(K_{n-1})$ are
$-(n-2), 1, (n-2)^2$. This implies that the least eigenvalue of $H^TH$
is
\[
\frac{(n-1)!}{2}-\frac{(n-2)(n-2)!}{2}>0.
\]
This proves that $H^TH$ is non-singular and hence full
rank. This, in turn, proves that $\{v_{i,j}\,|\, i,j\in[n-1]\}$ is
linearly independent.
\end{proof}

\section{Proof of The Main Theorem}\label{main_proof}

Define the $|G_n|\times n^2$ matrix $A$ to be the matrix whose columns
are the characteristic vectors $v_{i,j}$ of the sets $S_{i,j}$, for
all $i,j\in \{1,\dots,n\}$. Then since $A$ has constant row-sums, the vector
$\mathbf{1}$ is in the column space of $A$; thus in the light of
Lemma~\ref{basis}, we observe the following.
\begin{lem}\label{col_space_A}
The characteristic vector of any maximum intersecting subset of $G_n$ lies in the column space of $A$.
\end{lem}

We denote by $A_{i,j}$ the column of $A$ indexed by the pair $(i,j)$,
for any $i,j\in \{1,\dots,n\}$. Define the matrix $\overline{A}$ to be the matrix
obtained from $A$ by deleting all the columns $A_{i,n}$ and $A_{n,j}$
for any $i,j\in[n-1]$. Note that $\overline{A}$ is also obtained from
$H$ by adding the column $A_{n,n}$. With a similar method as in the
proof of \cite[Proposition 10]{MeagherS11}, we prove the following.
\begin{lem}\label{col_A_bar}
The characteristic vector of any maximum intersecting subset of $G_n$ lies in the column space of $\overline{A}$.
\end{lem}
\begin{proof}
  According to Lemma~\ref{col_space_A}, it is enough to show that the
  two matrices $A$ and $\overline{A}$ have the same column
  space. Obviously, the column space of $\overline{A}$ is a subspace
  of the column space of $A$; thus we only need to show that the vectors $A_{i,n}$
  and $A_{n,j}$ are in the column space of $\overline{A}$, for any
  $i,j\in[n-1]$. Since $G_n$ is two transitive, it suffices to show
  this for $A_{1,n}$. Define the vectors $v$ and $w$ as follows:
\[
v:=\sum_{i\neq 1,n} \sum_{j\neq n} A_{i,j}\quad\text{and}\quad w:=(n-3) \sum_{j\neq n} A_{1,j}\,+A_{n,n}.
\]
The vectors $v$ and $w$ are in the column space of $\overline{A}$. It is easy to see that for any $\pi \in G_n$,
\[
v_\pi=
\begin{cases}
n-2,& \quad\text{if}\quad \pi(1)=n\\
n-2, & \quad \text{if} \quad\pi(n)=n\\
n-3,& \quad\text{otherwise},
\end{cases}
\quad\quad\quad
w_\pi=
\begin{cases}
0,& \quad\text{if}\quad \pi(1)=n\\
n-2, &  \quad\text{if}\quad\pi(n)=n\\
n-3,& \quad\text{otherwise}.
\end{cases}
\]
Thus 
\[
(v-w)_\pi=
\begin{cases}
n-2,& \quad\text{if}\quad \pi(1)=n\\
0, & \quad \text{if} \quad\pi(n)=n\\
0,& \quad\text{otherwise},
\end{cases}
\]
which means that $(n-2)A_{1,n}=v-w$. This completes the proof.
\end{proof}

If the columns of $\overline{A}$ are arranged so that the first $n$
columns correspond to the pairs $(i,i)$, for $i\in \{1,\dots,n\}$, and
the rows are arranged so that the first row corresponds to the
identity element, and the next $|\mathcal{D}_{G_n}|$ rows correspond
to the elements of $\mathcal{D}_{G_n}$ (recall that these are the
derangements of $G_n$), then $\overline{A}$ has the following block
structure:
\[
\left[\begin{tabular}{cc}
1& 0\\
0 & M \\
B & C\\
\end{tabular}\right].
\]
Note that the rows and columns of $M$ are indexed by the elements of
$\mathcal{D}_{G_n}$ and the pairs $(i,j)$ with $i,j\in[n-1]$ and
$i\neq j$, respectively; thus $M$ is a $|\mathcal{D}_{G_n}|\times
(n-1)(n-2)$ matrix. We will next prove that $M$ is full rank.

\begin{prop}\label{fullrank}
For all $n\geq 5$, rank of $M$ is $(n-1)(n-2)$.
\end{prop}
\begin{proof} First assume $n$ is odd. Consider the submatrix $M_1$ of
  $M$ that is comprised of all the rows in $M$ that are indexed by
  cyclic permutations of length $n$.  Set $T=M_1^TM_1$; it suffices to
  show that $T$ is non-singular. Consider all types of entries of
  $T$. If $i,j,k,l$ are in $[n-1]$, then the following are all
  possible cases for the pairs $(i,j)$ and $(k,l)$.
\begin{itemize}
\item $i=k$ and $j=l$; in this case $T_{(i,j),(i,j)}=(n-2)!$; because
  the number of all $n$-cycles mapping $i$ to $j$ is $(n-2)!$.

\item $i=l$ and $j=k$; in this case $T_{(i,j),(j,i)}=0$; because the
  only case in which an $n$-cycle can swap $i$ and $j$ is $n=2$.

\item $i=k$ and $j\neq l$; in this case $T_{(i,j),(i,l)}=0$; because
  there is no permutation mapping $i$ to two different numbers.

\item $i\neq k$ and $j=l$; again $T_{(i,j),(k,j)}=0$.

\item $i\neq l$ and $j=k$; in this case $T_{(i,j),(j,l)}=(n-3)!$;
  because the number of all $n$-cycles mapping $i$ to $j$ and $j$ to
  $l$ is $(n-3)!$.

\item $i=l$ and $j\neq k$; in this case $T_{(i,j),(k,i)}=(n-3)!$; with
  a similar reasoning as above.

\item $\{i,j\}\cap \{k,l\}=\emptyset$; in this case
  $T_{(i,j),(k,l)}=(n-3)!$; because the number of $n$-cycles
  mapping $i$ to $j$ and $k$ to $l$ is ${n-3\choose 1}(n-4)!=(n-3)!$.
\end{itemize}

Therefore, one can write $T$ as
\begin{equation}\label{4th eq}
T=(n-2)!I+(n-3)!A(X),
\end{equation}
where $I$ is the identity matrix of size $(n-1)(n-2)$ and $A(X)$ is
the adjacency matrix of the graph $X$ defined as follows: the vertices
of $X$ are all the ordered pairs $(i,j)$ where $i,j\in[n-1]$ and
$i\neq j$; the vertices $(i,j)$ and $(k,l)$ are adjacent in $X$ if and
only if either $\{i,j\}\cap\{k,l\}=\emptyset$, or $i=l$ and $j\neq k$,
or $i\neq k$ and $j=l$. Note that $X$ is a regular graph of valency
$(n-2)(n-3)$. Our next result, Lemma~\ref{least_eval_of_X}, will show that the
least eigenvalue of $X$ is greater than or equal to $-(n-3)$; thus
using (\ref{4th eq}), the least eigenvalue of $T$ is at least
\[(n-2)!-(n-3)!(n-3)=(n-3)!>0;
\]
therefore $T$ is non-singular and the proof is complete for the case $n$ is odd.

Now assume $n$ to be even.  Consider the subset of
$\mathcal{D}_{G_n}$ which consists of all the permutations of $G_n$
whose cycle decomposition includes two cycles of length $n/2$ and let
$M_2$ be the submatrix of $M$ whose rows are indexed by
these permutations. Define $U=M_2^TM_2$. With a similar approach as
for the previous case, one can write $U$ as
\[
U=\frac{2(n-2)!}{n}I+\frac{2(n-3)!}{n}A(X).
\]
According to Lemma~\ref{least_eval_of_X}, the least eigenvalue of $U$ is at least
\[ 
\frac{2(n-2)!}{n}-\frac{2(n-3)!}{n}(n-3)=\frac{2(n-3)!}{n}>0;
\]
therefore $U$ is non-singular and the proof is complete.
\end{proof}

\begin{lem}\label{least_eval_of_X}
 Let $n>3$ and $X$ be the graph defined in Lemma~\ref{fullrank}. The least eigenvalue of $X$ is at least $-(n-3)$.
\end{lem}
\begin{proof}

  First note that any cyclic permutation $\alpha=(i_1,\ldots,
  i_{n-1})$ of $\{1,2,\dots ,n-1\}$ corresponds to a unique clique of
  size $n-1$ in $X$; namely the clique $C_\alpha$ induced by the
  vertices $\{(i_1,i_2), (i_2,i_3),\ldots, (i_{n-2},i_{n-1}),
  (i_{n-1},i_1)\}$. We claim that any edge of $X$ is contained in
  exactly $(n-4)!$ cliques of form $C_\alpha$. 

  Consider the edge $\{(a,b),(c,d)\}$.  If $a\neq d$ and $b=c$ (or
  $a=d$ and $b \neq c$), then there are $(n-4)!$ cyclic permutation of
  form $(a,b,d, -, - ,\cdots, -)$ and this edge is in exactly $(n-4)!$
  of the cliques. If $\{a,b\}\cap \{c,d\}=\emptyset$ then there are
  again $(n-4)!$ cyclic permutation of form $(a,b, -,\cdots,-, c,d,
  -,\cdots,-)$ (as there are $(n-4)$ ways to assign a position for the
  pair $c,d$, and then there are $(n-5)!$ ways to arrange other
  elements of $\{1,\ldots,n-1\}$ in the remaining spots).

  Thus the claim is proved. If $\tau$ denotes the least
  eigenvalue of $X$, then we can apply Proposition~\ref{w} to $X$ to
  get that
\[
\tau\geq \frac{-k}{w-1}=-\frac{(n-2)(n-3)}{n-2}=-(n-3).\qedhere
\]
\end{proof}

Now we are ready to prove the main theorem.

\begin{proof} (Theorem~\ref{main_Alt}) Let $S$ be a intersecting set
  of permutations in $G_n$. Then $S$ is an independent set in the
  graph $\Gamma_{G_n}$. In Subsections~\ref{subsec:odd} and
  \ref{subsec:even} cliques of size $n$ in $\Gamma_{G_n}$ are
  given. By the clique-coclique bound no independent set is larger
  than $\frac{(n-1)!}{2}$; thus the bound in Theorem~\ref{main_Alt}
  holds.

  Further suppose that $S$ is of maximum size (namely
  $\frac{(n-1)!}{2}$) and let $v_S$ be the characteristic vector of
  $S$.  To complete the proof of this theorem, it is enough to show
  that $S=S_{i,j}$, for some $i,j\in \{1,\dots,n\}$.

  Without loss of generality, we may assume that $S$ includes the
  identity element. By Lemma~\ref{col_A_bar}, $v_S$ is in the column
  space of $\overline{A}$, thus
\[
\left[\begin{tabular}{cc}
1 & 0\\
0 & M \\
B & C\\
\end{tabular}\right]\begin{bmatrix} v \\ w  \end{bmatrix}
=v_S
\]
for some vectors $v$ and $w$. Since the identity is in $S$, no
elements from $D_{G_n}$ are in $S$, and the characteristic vector of
$S$ has the form
\[
v_S= \begin{bmatrix} 1 \\ 0 \\ t  \end{bmatrix}
\]
for some vector $t$. Thus we have $1^T v=1$, $Mw=0$ and
$Bv+Cw=t$. According to Lemma~\ref{fullrank}, $M$ is full rank;
therefore, $w=0$ and so $Bv=t$. 

Furthermore, for any $x\in \{1,\dots,n\}$, there is a permutation $g_x\in G_n$
which has only $x$ as its fixed point. Then by a proper permutation of
the rows of $B$, one can write
\[
B=\begin{bmatrix} I_{n} \\[.2cm] B'  \end{bmatrix}\quad 
\text{and}\quad 
Bv=\begin{bmatrix} v \\[.2cm] B'v  \end{bmatrix}. 
\]
Since $Bv$ is equal to the $01$-vector $t$, the vector $v$ must also
be a $01$-vector. But, on the other hand, $1^T v=1$, thus we conclude that
exactly one of the entries of $v$ is equal to $1$. This means that $v_S$ is the
characteristic vector of the stabilizer of a point.
\end{proof}


\section{Conclusions}\label{conclusions}

An interesting result of Theorem~\ref{main_Alt} is that it implies
that the symmetric group also has the strict EKR property. To show
this we will state two results that were first pointed out by Pablo
Spiga; the proof we give of the first result is due to Chris
Godsil.

\begin{theorem}
  Let $G$ be a transitive subgroup of $\sym(n)$ and let $H$ be a
  transitive subgroup of $G$. If $H$ has the EKR property, then $G$
  has the EKR property. 
\end{theorem}
\proof The group $H$ has the EKR property and is transitive, so the
size of the maximum coclique is $|H|/n$. Further, the graph $\Gamma_H$
is vertex transitive so its fractional chromatic number is $n$ (see
\cite[Chapter 7]{MR1829620} for details about the fractional chromatic
number of a graph).
 
The embedding $\Gamma_H \rightarrow \Gamma_G$ is a homomorphism, so
the fractional chromatic number of $\Gamma_G$ is at least the
fractional chromatic number of $\Gamma_H$. The graph
$\Gamma_G$ is also vertex transitive, so
\[
n \leq \frac{|G|}{\alpha(\Gamma_G)}
\]
where $\alpha(\Gamma_G)$ is the size of a maximum independent set.
Thus $\alpha(\Gamma_G) \leq \frac{|G|}{n}$, and since $G$ is
transitive the stabilizer of a point achieves this bound.
\qed

If the groups $G$ and $H$ in the above theorem are $2$-transitive then
we can say something about when the strict EKR property holds for $G$.

\begin{thm}\label{2transitive_subgroup} Let $G$ be a $2$-transitive
  subgroup of $\sym(n)$ and let $H$ be a $2$-transitive subgroup of
  $G$. If $H$ has the strict EKR property then $G$ has the strict EKR
  property.
\end{thm}
\begin{proof}
  Since $H$ has the strict EKR property, it also has the EKR property
  and by the previous result $G$ also has the EKR property. Assume
  that $S$ is a coclique in $\Gamma_G$ of size $|G|/n$ that contains
  the identity; we will prove that $S$ is the stabilizer of a point.

  Let $\{x_1=id,\dots,x_{[G:H]}\}$ be a left transversal of $H$ in $G$
  and set $S_i = S \cap x_iH$.  Then for each $i$ the set
  $x_i^{-1}S_i$ is an independent set in $\Gamma_H$ with size $|H|/n$. Since
  $H$ has the strict EKR property each $x_i^{-1}S_i$ is the coset of a
  stabilizer of a point.

  Since $x_1 = id$, the identity is in $S_1$ which means that $S_1$ is
  the stabilizer of a point and we can assume that $S_1 = H_\alpha$
  for some $\alpha \in \{1,\dots,n\}$. We need to show that every
  permutation in $S$ also fixes the point $\alpha$. Assume that there
  is a $\pi \in S$ that does not fix $\alpha$.  Since $S$ is
  intersecting, for every $\sigma \in S_1$ the permutation $\sigma
  \pi^{-1}$ fixes some element (but not $\alpha$ and not
  $\pi(\alpha)$), from this it follows that
\[
H_\alpha \pi^{-1} \subseteq \bigcup_{\stackrel{\beta \neq \alpha}{\beta \neq \pi(\alpha)}} G_\beta
= \bigcup_{\stackrel{\beta \neq \alpha}{\beta \neq \pi(\alpha)}} ( G_\beta \cap H_\alpha \pi^{-1} ).
\]

Assume that $\sigma \pi^{-1} \in G_\beta \cap H_\alpha \pi^{-1}$, then
$\beta^{\sigma \pi^{-1}} = \beta$ and $\alpha^{\sigma} = \alpha$.  The
permutation $\sigma \pi^{-1}$ must map $(\alpha, \beta)$ to
$(\alpha^{\pi^{-1}}, \beta)$.  Since the group $H$ is $2$-transitive
there are exactly $|H|/n(n-1)$ such permutations and we have that
\[
| G_\beta \cap H_\alpha \pi^{-1}| = \frac{|H|}{n(n-1)}.
\]
From this we have that the size of $H_\alpha \pi^{-1}$ is
\[
\sum_{\stackrel{\beta \neq \alpha}{\beta \neq \pi(\alpha)}}\frac{|H|}{n(n-1)} = (n-2) \frac{|H|}{n(n-1)},
\]
but since this is strictly less that $\frac{|H|}{n}$,  this is a contradiction.
\end{proof}

This theorem along with Theorem~\ref{main_Alt} provides an alternative
proof of the following theorem which was initially proved in
\cite{MR2009400}.
\begin{cor} 
For any $n\geq 2$, the group $\sym(n)$ has the strict EKR property.
\end{cor}

\end{document}